\documentclass{amsart}
\usepackage[left=1in,right=1in,top=1in,bottom=1in]{geometry}
\usepackage{comment}
\usepackage{amsmath, amsthm, amssymb,graphicx,float,mathtools,verbatim, subcaption, color}
\usepackage[toc,page]{appendix}

\usepackage{epsfig,cancel,amsmath,amsthm,amssymb,wrapfig,subcaption, makecell, mathrsfs, hyperref}
\usepackage{bbm,bm}
\usepackage{graphicx, xcolor}
\newcommand{\J}{\mathcal{J}}

\newtheorem{theorem}{Theorem}[section]

\newtheorem{lemma}[theorem]{Lemma}
\newtheorem{remark}[theorem]{Remark}

\newtheorem{proposition}[theorem]{Proposition} 
 
\newtheorem{definition}[theorem]{Definition}

\newcommand{\e}{\text{exp}}

\newcommand{\la}{\langle}
\newcommand{\ra}{\rangle}
\newcommand{\pd}{\partial}

\newcommand{\ud}{\,\mathrm{d}}
\newcommand{\I}{\mathbb{T}}
\newcommand{\8}{\infty}

\newcommand{\noin}{\noindent}
\newcommand{\beq}{\begin{equation}}
\newcommand{\eeq}{\end{equation}}

\newcommand{\Z}{\mathbb Z}

\def\beqs#1\eeqs{%
    \begin{equation}\begin{split}%
    #1%
    \end{split}\end{equation}%
}

\title[PDE for Crystal Surface Models]{Analysis of a fourth order exponential PDE arising from a crystal surface jump process with Metropolis-type transition rates}
\author[Y. Gao]{Yuan Gao}
\address{Department of Mathematics, Duke University, Durham, NC 27708}
\email{yuangao@math.duke.edu}

\author[A.E. Katsevich]{Anya E. Katsevich}
\address{Courant Institute of Mathematical Sciences, New York University\\
251 Mercer Street, New York, NY 10012, USA}
\email{katsevich@cims.nyu.edu}

\author[J.-G. Liu]{Jian-Guo Liu}
\address{Department of Mathematics and Department of Physics, Duke University, Durham, NC 27708}
\email{jliu@math.duke.edu}

\author[J. Lu]{Jianfeng Lu}
\address{Department of Mathematics, Department of Physics, and Department of Chemistry, Duke University, Durham, NC 27708 }
\email{jianfeng@math.duke.edu}

\author[J.L. Marzuola]{Jeremy L. Marzuola}
\address{Department of Mathematics, UNC-Chapel Hill \\ CB\#3250
  Phillips Hall \\ Chapel Hill, NC 27599}
\email{marzuola@math.unc.edu}

\begin{document}    

\begin{abstract}

We analytically and numerically study a fourth order PDE modelling rough crystal surface diffusion on the macroscopic level. We discuss existence of solutions globally in time and long time dynamics for the PDE model. The PDE, originally derived by the second author, is the continuum limit of a microscopic model of the surface dynamics, given by a Markov jump process with Metropolis type transition rates. We outline the convergence argument, which depends on a simplifying assumption on the local equilibrium measure that is valid in the high temperature regime. We provide numerical evidence for the convergence of the microscopic model to the PDE in this regime.



\end{abstract}
   
\maketitle 


\section{Introduction}
In this paper, we analyze a fourth order exponential PDE which models the macroscopic dynamics of crystal surface relaxation. The PDE arises as the time and space scaling limit of a microscopic Markov jump process, which evolves via discrete surface hopping events. These events occur at specified transition rates. The transition rates, along with the form of the interaction potential, determine the microscopic dynamics and therefore also shape the macroscopic PDE. (The choice of interaction potential affects the Hamiltonian, through which the equilibrium probability distribution is defined). Here we focus on a quadratic interaction potential and transition rates of Metropolis type, meaning that they are a function only of the difference in energy between the pre- and post-jump crystal states.  Below we will restrict ourselves to one spatial dimension with periodic boundary conditions for simplicity of exposition, but higher dimensional generalizations and other types of boundary condition are indeed possible and treatable using the presented techniques.

In the continuum limit, the crystal surface is represented as a height profile $h(t,x), t\geq0, x\in\mathbb T,$ where $\mathbb T$ is the unit interval with periodic boundary conditions. The resulting PDE limit is of the form
\begin{equation}\label{anyaexpde}
\partial_t h(t,x) = \partial_x \left(\e\left[-\beta\partial_x^3 h(t, x)\right] - \e\left[\beta\partial_x^3 h(t,x)\right]\right), \ \ h(0,x) = h_0 (x),
\end{equation} where $\beta$ denotes inverse temperature. Conditions on the initial data required for the existence theory will be discussed below. In this paper, we develop analytic tools for the PDE \eqref{anyaexpde}. In particular, we prove the existence of global solutions to \eqref{anyaexpde} in Section \ref{sec:longtime}.  We also explore properties of \eqref{anyaexpde} numerically in Section \ref{sec:numerics}.

This PDE has several key features that distinguish it from analogous exponential PDEs derived as the macroscopic limit of Arrhenius rate dynamics in \cite{mw-krug, krug1995adatom}, and recently revisited in \cite{gao2019analysis}. The Arrhenius rates are another set of rates for microscopic dynamics which are used in the Solid-On-Solid (SOS) model, a well-known (and well-studied) model for crystal surface relaxation \cite{Binh}. 
The Arrhenius rate PDE in \cite{krug1995adatom} assumes an absolute value potential, while the PDE in \cite{mw-krug} is derived using a more general interaction potential, including in particular the quadratic one. For the quadratic interaction potential, it is given by an equation of the form (taking $\beta=1$ for simplicity)
\begin{equation}
\label{Hm1expde}
\partial_t h = \partial_{xx}\, \e\left(- \partial_{xx} h\right), \ \ h(0,x) = h_0(x).
\end{equation}
    This PDE can be interpreted as a weighted $H^{-1}$ gradient flow in a similar fashion to the framework laid out in the related works \cite{liu2017asymmetry,craig2020proximal}.
Recent analytic progress has been made relating to existence of weak solutions, characterization of dynamics, construction of strong solutions, and classification of the breakdown of regularity for this equation. See for instance the works \cite{liu2016existence,liu2017analytical,liu2017asymmetry,gao2017gradient,xu2017existence}. 


One symmetry property of \eqref{anyaexpde} not present in \eqref{Hm1expde} demonstrates an important intrinsic distinction between the microscopic Metropolis and Arrhenius rate dynamics. Namely, if $h$ is a solution to \eqref{anyaexpde} with initial profile $h_0$, then $-h$ is another solution with corresponding initial data $-h_0$. This is not the case for \eqref{Hm1expde}. In fact, solutions to \eqref{Hm1expde} form singularities in convex regions but not in concave regions \cite{mw-krug}. Using the structure of the Metropolis and Arrhenius rates, one can show this same symmetry property (or lack thereof) holds on a microscopic level. 

We note that the exponential PDEs \eqref{anyaexpde} and \eqref{Hm1expde} arise by applying a nonstandard but potentially more informative scaling regime to the microscopic dynamics (i.e. relative scaling of time, space, and height in the large crystal limit). In a more standard scaling regime, the exponentials become linearized. 

In~\cite{mw-krug}, the limiting macroscopic PDEs in both scaling regimes are derived using a probabilistic framework. These PDEs can also be derived using physical arguments. In \cite{krug1995adatom}, for example, the authors use physical principles to derive a PDE for the Arrhenius rate broken bond model with an absolute value potential. The resulting PDE does not have exponential dependence on the derivatives of $h$, and corresponds to the standard scaling regime. A separate argument in the last section of \cite{krug1995adatom} suggests an exponential PDE similar to \eqref{Hm1expde} as an alternative. In \cite{krug1995adatom} as well as in  \cite{gao2019analysis}, the exponential PDE is obtained by applying the so-called Gibbs-Thomson relation as an underlying approximation for how the density varies with respect to the chemical potential; see \cite{liu2017asymmetry}, Section $1.1$ for details on this approach.

The derivation of~\eqref{anyaexpde} uses the same probabilistic framework as~\cite{mw-krug}. The argument is exactly the same, except of course that one should replace the Arrhenius rates with the Metropolis rates. For this reason, and because the analysis of exponential PDE \eqref{anyaexpde} is the primary focus of this paper, we only give a high level overview of the argument. The argument relies on the assumption that in local equilibrium, the process's distribution is given by a local Gibbs measure. This assumption is correct for the Arrhenius rates but is not entirely accurate for the Metropolis rates due to their jump asymmetry. As a consequence, the PDE~\eqref{anyaexpde} is not necessarily accurate in general; see the end of Section~\ref{numerics} for further discussion.  Nevertheless, the equation is an excellent fit for the microscopic dynamics when $\beta$ is small. We add that the equation retains its nonlinearity for small $\beta$ (i.e. one cannot linearize the exponentials), as evidenced by the evolution of numerical solutions to~\eqref{anyaexpde} shown in Section~\ref{numerics}.

The paper will proceed as follows.  In Section \ref{sec:micsp}, we describe the microscopic process and scaling regime leading to the PDE~\ref{anyaexpde}. We then compare the evolution of the PDE solution to that of the microscopic process. The section concludes with a discussion of the slight discrepancy between the PDE dynamics and the true large crystal microscopic dynamics. In Section \ref{sec:longtime}, we prove the global existence and long time behavior of the solutions to  PDE \eqref{Hm1expde}.  In Section \ref{sec:numerics}, we explore properties of the PDE numerically. In Appendix \ref{rate-expect} we compute the Metropolis rate expectation with respect to the local Gibbs measure. This is the main computation needed to derive the PDE from the microscopic dynamics. In Appendix \ref{A:LTalt}, we give an alternative approach to the analysis of long time dynamics for \eqref{anyaexpde}.

\subsection*{Acknowledgements} This project was started while JLM was
on sabbatical at Duke University in the Spring of 2019.  JLM thanks
Bob Kohn, Dio Margetis and Jonathan Weare for many
valuable conversations regarding modeling of kinetic Monte Carlo. AEK is supported by the DOE Computational Science Graduate Fellowship. JGL
was supported by the National Science Foundation (NSF) grant
DMS-1812573 and the NSF grant RNMS-1107444 (KI-Net).  JL was supported
by the National Science Foundation via grant DMS-1454939.  JLM
acknowledges support from the NSF through NSF CAREER Grant
DMS-1352353 and NSF grant DMS-1909035. 

\section{Motivation: PDE as Scaling Limit}
\label{sec:micsp}
The PDE~\eqref{anyaexpde} arises as the continuum limit of a discrete microscopic Markov jump process modeling the relaxation of a crystal surface. We briefly describe the state space and dynamics of this microscopic process. 
The process is represented by a height profile $$h_N(t)=(h_N^0(t),\dots, h_N^{N-1}(t))\in\Z^N, t\geq 0.$$ For fixed $t$, we may think of $h_N(t)$ as a step function in space on the torus $\mathbb T=[0,1)$ (with endpoints identified), which takes the value $h_N^j(t)$ on $[j/N, (j+1)/N)$. We call the interval $[j/N, (j+1)/N)$ ``site $j$''. The integer values $h_N^j$ represent the number of particles at site $j$ stacked in a column, above or below a fixed level surface representing zero height. 

The process evolves by means of particles jumping between neighboring sites. Each such jump occurs instantaneously at a certain transition or jump rate. The jumps and their corresponding jump rates fully determine the dynamics and must be specified in advance. 

We denote the event  in which a particle jumps from site $i$ to site $j$ (with $|i-j|=1$) by $h\mapsto J_i^jh$. That is, if the profile was given by $h_N(t)=h$ before the jump, then after the jump it is given by $J_i^jh$, where
\beqs
&(J_i^jh)^i= h^i-1,\quad (J_i^jh)^j= h^j+1,\\
&(J_i^jh)^k= h^k, \quad k\neq i,j.
\eeqs 

The rate $r^{i,j}(h)$ at which this transition occurs is defined through 
$$\mathbb P\bigg(h_N(t+\delta)  = J_i^jh\;\bigg\vert\; h_N(t)=h\bigg) =  r^{i,j}(h)\delta + o(\delta).$$ In particular, we assume $r^{i,j}(h)$ is independent of time. The jump rates determine the expected instantaneous change in height induced by a jump. Thus while a jump from site $i$ to $j$ increases $h_N^j$ by $1$, the expected increase in $h_N^j$ is given by $r^{i,j}(h_N(t))dt$.  

Let us write down the evolution of $h_N^i$ at a given moment in time. It will decrease if a jump occurs from $i$ to $i+1$ or from $i$ to $i-1$, and increase if a jump occurs from $i+1$ or $i-1$ to $i$. We therefore obtain
$$dh_N^i(t)= \bigg[(r^{i-1,i}-r^{i,i-1})(h_N(t)) - (r^{i,i+1}-r^{i+1,i})(h_N(t))\bigg]dt + d\xi_N^i(t),$$ where $\xi_N^i$ represents random fluctuations. 

If we define $\J_N^i = r^{i,i+1}-r^{i+1,i}$ as the ``current'' from $i$ to $i+1$, then we can write the above as a microscopic conservation law:
$$dh_N^i(t)= -(\J_N^i-\J_N^{i-1})(h_N(t))dt + d\xi_N^i(t).$$

By scaling time and height with $N$, the random fluctuations vanish and we obtain the deterministic limit $$h(t,x) = \lim_{N\to\infty}N^{-3}h_N^{Nx}(N^4t),$$ with 
\beq\label{J-eqn}\partial_th = -\partial_x\J^x(h).\eeq Here $\J^x(h)= \lim_{N\to\infty}\langle \J_N^{Nx}(h_N)\rangle,$ where $\langle\cdot\rangle$ denotes expectation with respect to the local Gibbs measure (see the Appendix for a definition of the measure). 

We note that the above argument is only heuristic. In particular, we do not necessarily expect pointwise convergence of the $h_N^i$ to $h(t,x)$. One can only expect the convergence of averages of $h_N^i$ in mesoscopic-sized intervals to the macroscopic profile. 

Note that the equation~\eqref{anyaexpde} is of the form~\eqref{J-eqn}, with
$$\J^x(h) = e^{-\frac32\beta}\sinh(\beta h_{xxx}(t,x)).$$ It is the scaling limit of the microscopic process under a certain choice of transition rates, which we now describe. Namely, we set
\beq\label{metrop-def} r^{i,j}(h) = \e\left(-\frac{\beta}{2}\bigg[H(J_i^jh)-H(h)\bigg]\right),\quad |i-j|=1,\eeq where $\beta$ is an inverse temperature and $H(h)$ is the Hamiltonian, i.e. surface energy, of a configuration $h$, defined by
$$H(h) = \sum_{i=0}^{N-1}(h^{i+1}-h^i)^2.$$ Note that the energy is independent of how the zero height level surface is chosen, i.e. a uniform shift $h^i\mapsto h^i+c,i=0,\dots,N-1$ does not affect the energy. 

These rates give preference to atomistic motion that lowers the surface energy $H$. Importantly, the dynamics is not symmetric with respect to jumping left or right, instead favoring the direction yielding lower energy.

The rates are in detailed balance with respect to the Gibbs distribution $$p_\infty(h)\propto e^{-\beta H(h)},$$ so that the Gibbs distribution is the global equilibrium ($t\to\infty$) distribution of the process. These rates belong to a class of rates of what we call  ``Metropolis type'', in that they are in detailed balance with the Gibbs measure and depend only on the energy difference $\Delta H$ between the pre- and post-jump states. Indeed, any rates of the form $\phi(\Delta H)$ which satisfy
$$\phi(-\Delta H) = \phi(\Delta H)e^{\beta\Delta H}$$ are in detailed balance with the Gibbs measure. The more well-known set of rates of Metropolis type are $\phi(\Delta H) = e^{-\beta\Delta H}\wedge1$. We note that while this is also the acceptance probability in standard Metropolis-Hastings algorithms, the goal of that algorithm is to sample the invariant Gibbs measure, whereas we are interested in the Markov process dynamics itself.  We choose rates of the form $e^{-\beta\Delta H/2}$ due to their analytic tractability, since it is more straightforward to compute the rate expectations $\langle r^{i,j}(h)\rangle$ for these rates than for the Metropolis rates involving a minimum.  

See Appendix~\ref{rate-expect} for the computation of $\langle r^{i,i+1}-r^{i+1,i}\rangle$, from which the PDE follows. 
\subsection{Scaling Limit: Numerics}\label{numerics}
In Figures~\ref{two-bump} and~\ref{sine-met}, we numerically compare the microscopic evolution with that of the solution to the PDE~\eqref{anyaexpde}. 

To solve the PDE, we discretized spatially using centered difference schemes and applied a numerical ODE timestepper designed for stiff ODEs.  To simulate the microscopic process, we used the Kinetic Monte Carlo (KMC) method. In this method, one iteratively updates the state of the process, $\bf h\mapsto  \bf h'$, and the physical time, $t\mapsto t + \Delta$ until the desired final time is reached.  The new state ${\bf h}'$ is randomly chosen with probability proportional to the transition rate $r\left({\bf h}\mapsto  {\bf h}'\right)$, and $\Delta$, which represents the amount of time the process spent in state ${\bf h}$, is chosen from an exponential distribution $\mathrm{Exp}(\lambda)$, with $\lambda=\sum_{{\bf h}'}r(\bf h\mapsto  \bf h').$ 

To compare the macroscopic and rescaled microscopic processes, we fix an initial non-trivial (out of equilibrium) macroscopic profile $h(x,0) = h_0(x)$ for which $||h||_{\infty} > 0$ in order to ensure non-trivial dynamics in both the microscopic and macroscopic flows. We evolve the PDE forward from $h_0$ to some macroscopic times of interest $T$. Then, for various $N$, we run KMC from the initial microscopic profile ${\bf h}_N(0) = \left(N^{3}h_0(j/N)\right)_{j=0}^{N-1}$ up to the microscopic time $N^4T$. We should then expect to see that $h(j/N, T) \approx N^{-3}h^j_N(N^4T)$ for large $N$. 

In addition to comparing the evolution of the KMC and PDE height profiles, we also check whether the Metropolis rate expectation computed using the local Gibbs measure converges to the true rate expectation. If this is the case, then one should have
\beq\label{numerical_rate_expect}\langle r^{i,i+1}(h_N(t))\rangle \approx \frac{1}{2N^4\delta}\int_{N^4(T-\delta)}^{N^4(T+\delta)}r^{i,i+1}(h_N(s))ds,\eeq
where the left hand side is the rate expectation with respect to the local Gibbs measure; see Appendix~\ref{rate-expect} for this expression. The expectation depends on the macroscopic profile $h(t,x)$, which we estimate using the PDE solution. The right hand side is the integral of a step function, since the microscopic process is a Markov jump process. It can therefore be simply computed from the KMC simulation by keeping track of the rate values and time between jumps. 

We consider two initial profiles and values of $\beta$: in Figure~\ref{two-bump}, we take $h_0(x) = g(x) + g(x+0.2 \mod 1)$ with $\beta = 0.01$, where $$g(x) = \frac{1}{10}\text{exp}\left(8 - 1/x -1/(\frac12 - x)\right)\mathbbm{1}_{(0,\frac12)}(x)$$ is a smooth bump function supported on $(0,\frac12)$. In Figure~\ref{sine-met}, $h_0(x) = \frac{1}{10}\sin(2\pi x)$, with $\beta = 0.25$. The reason for choosing small amplitude for the initial profile and small $\beta$ is to limit how large the rates can be, since they depend exponentially on $\beta$ and the curvature of ${\bf h}_N$. 

\begin{figure}
\includegraphics[width = 0.495\linewidth]{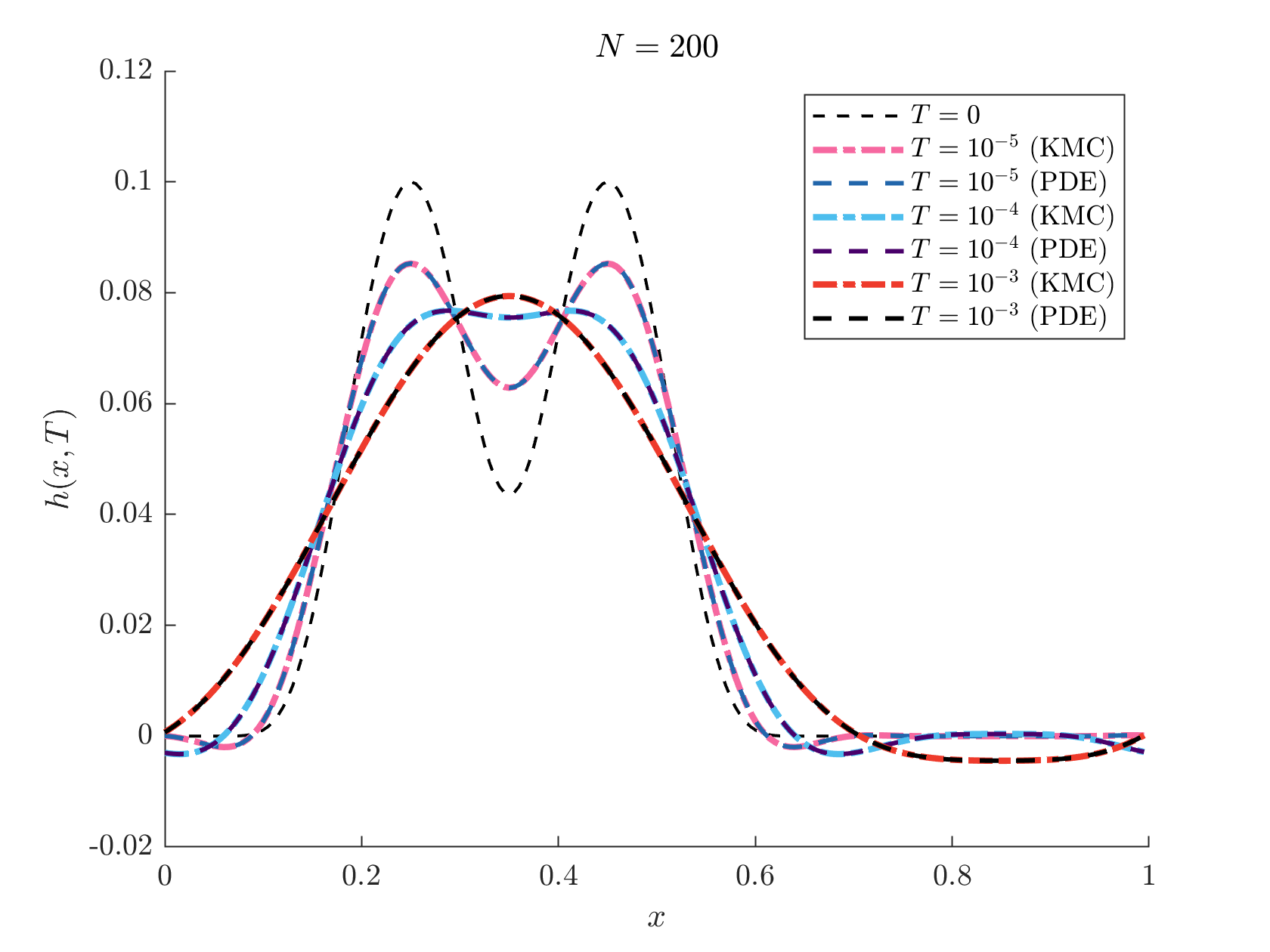}
\includegraphics[width = 0.495\linewidth]{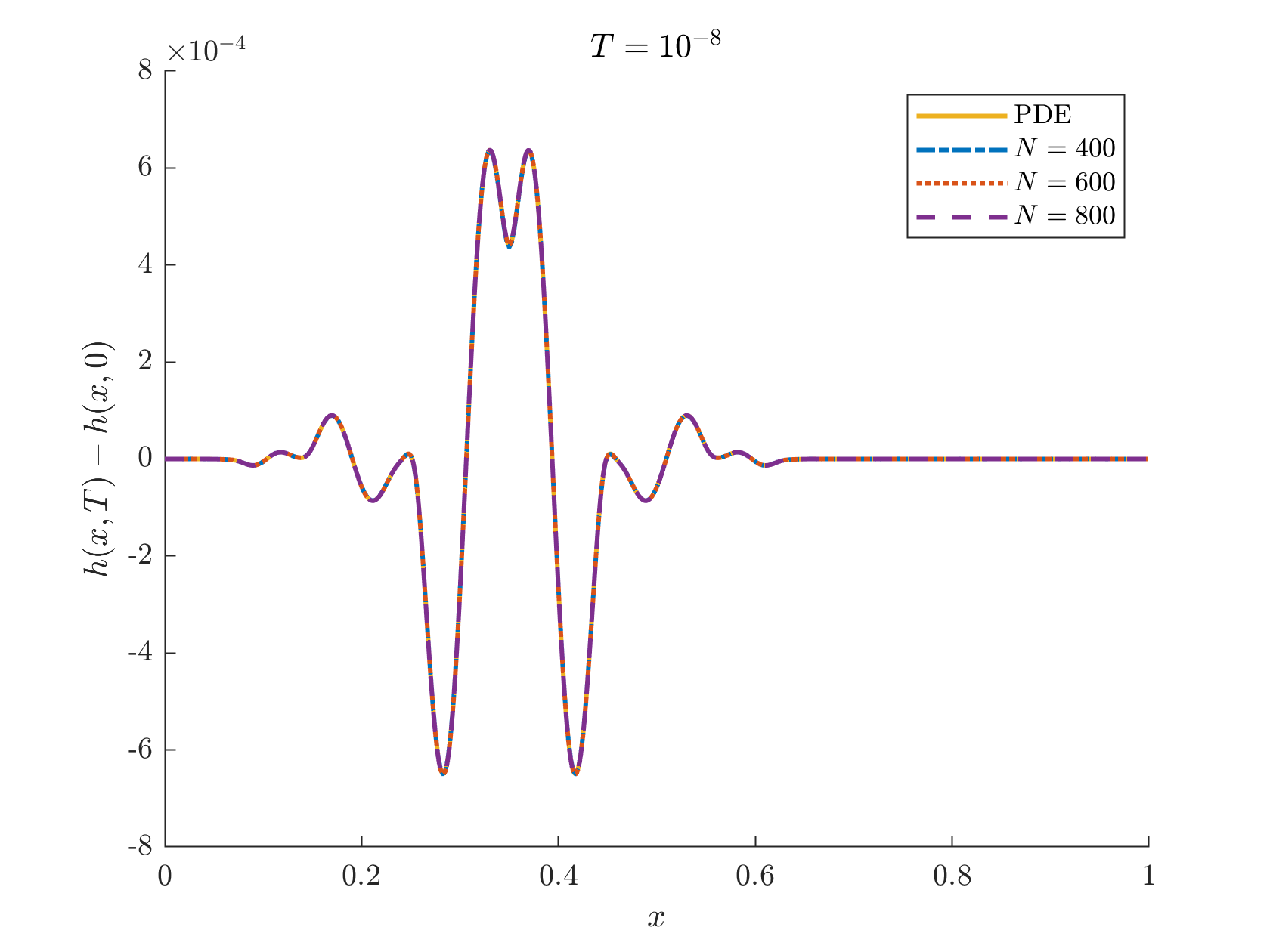}
\includegraphics[width = 0.495\linewidth]{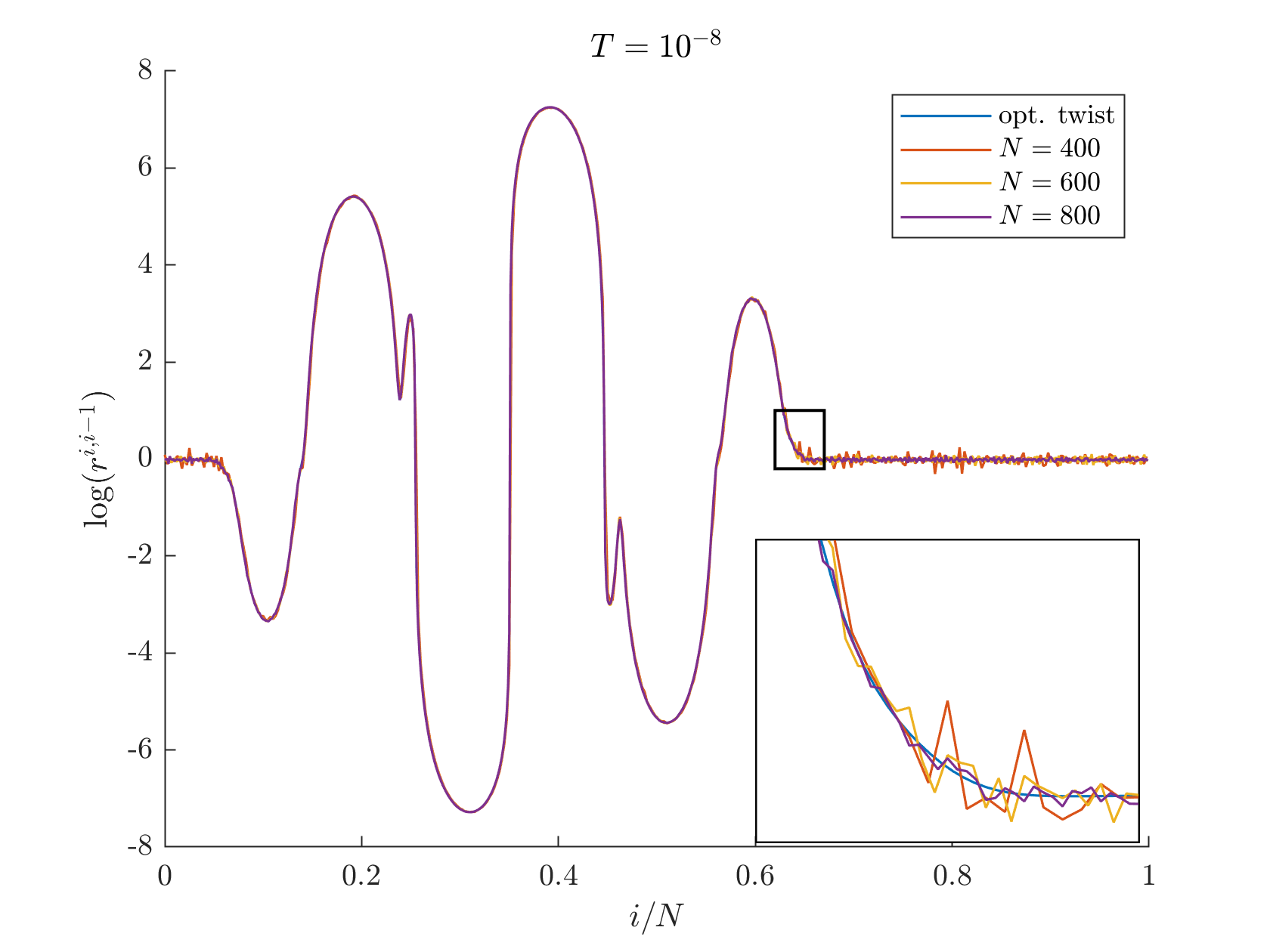}
\includegraphics[width = 0.495\linewidth]{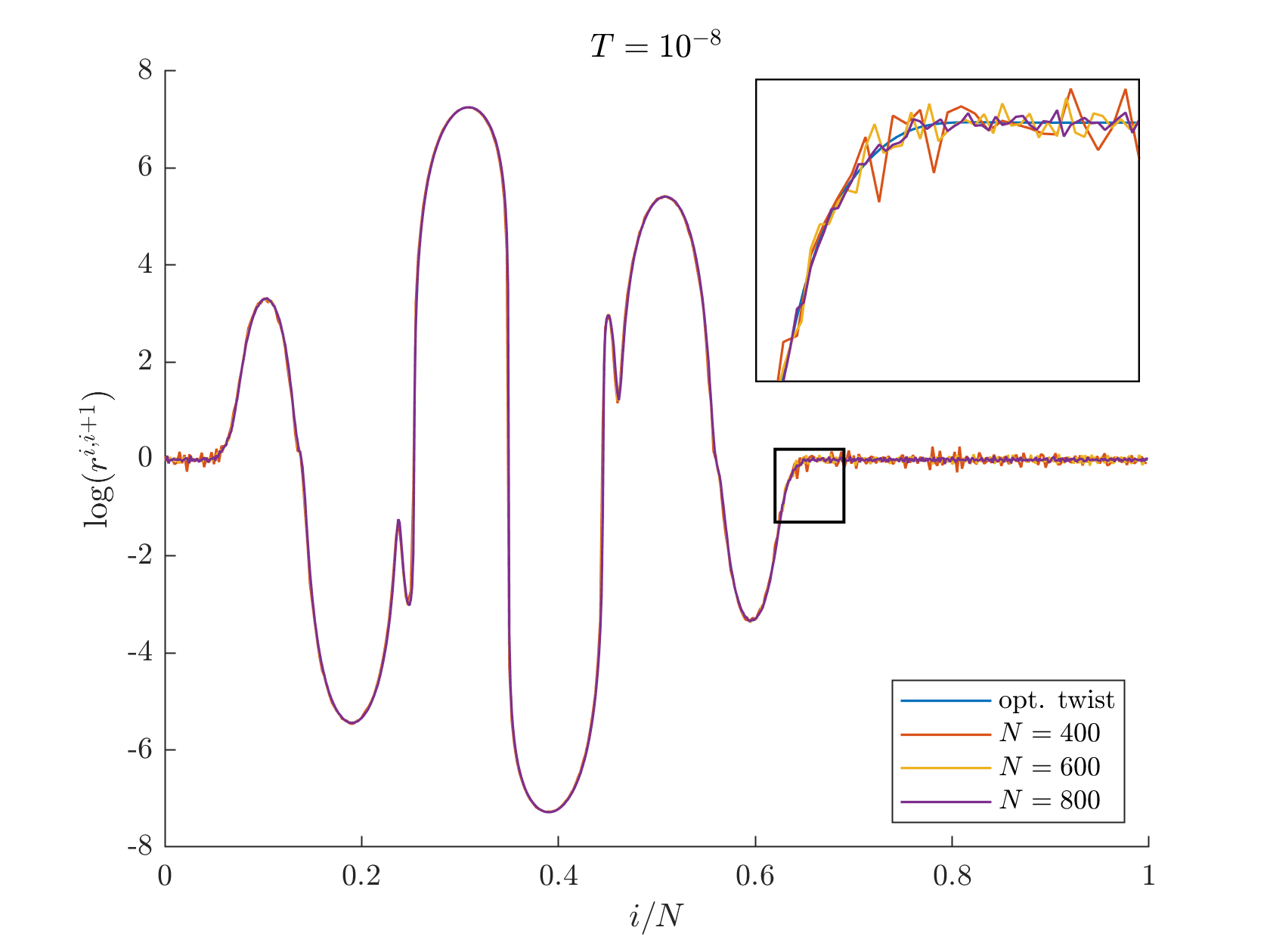}
\caption{Metropolis rate dynamics with temperature $\beta = 0.01$, from initial profile $h(x,0)$ given by a sum of a bump function $g(x)$ and its  translation, $g(x+0.2 \mod 1)$. Top left: PDE vs KMC evolution with fixed $N$ at various times $T$. Top right: PDE vs KMC short time evolution, $h(x,T) - h(x,0)$, for fixed $T$ and various $N$. Bottom left (bottom right): log of local Gibbs average vs log of empirical average of rates to jump left (jump right) for fixed $T$ and various $N$.}
\label{two-bump}
\end{figure}

Figure \ref{two-bump} shows results from the experiment with a compactly supported two-bump initial profile, and $\beta = 0.01$. We observe an excellent fit of the PDE solution to the microscopic profile obtained from KMC. The top left figure shows that on the scale of the initial height amplitude, the PDE dynamics exactly coincides with the microscopic dynamics for $N=200$. The top right figure shows that the PDE also fits the microscopic dynamics on the scale of small shifts in amplitude. Moreover, the microscopic dynamics (after rescaling) has already converged for $N = 400$, since increasing $N$ does not affect the dynamics. 

The bottom left (bottom right) figure compares the time average of the right-jump rates (left-jump rates) with its expectation with respect to the local Gibbs measure. The zoomed-in part of the plot shows that the KMC rate time average oscillates more closely around the local Gibbs rate expectation as $N$ increases.  

\begin{figure}
\includegraphics[width = 0.495\linewidth]{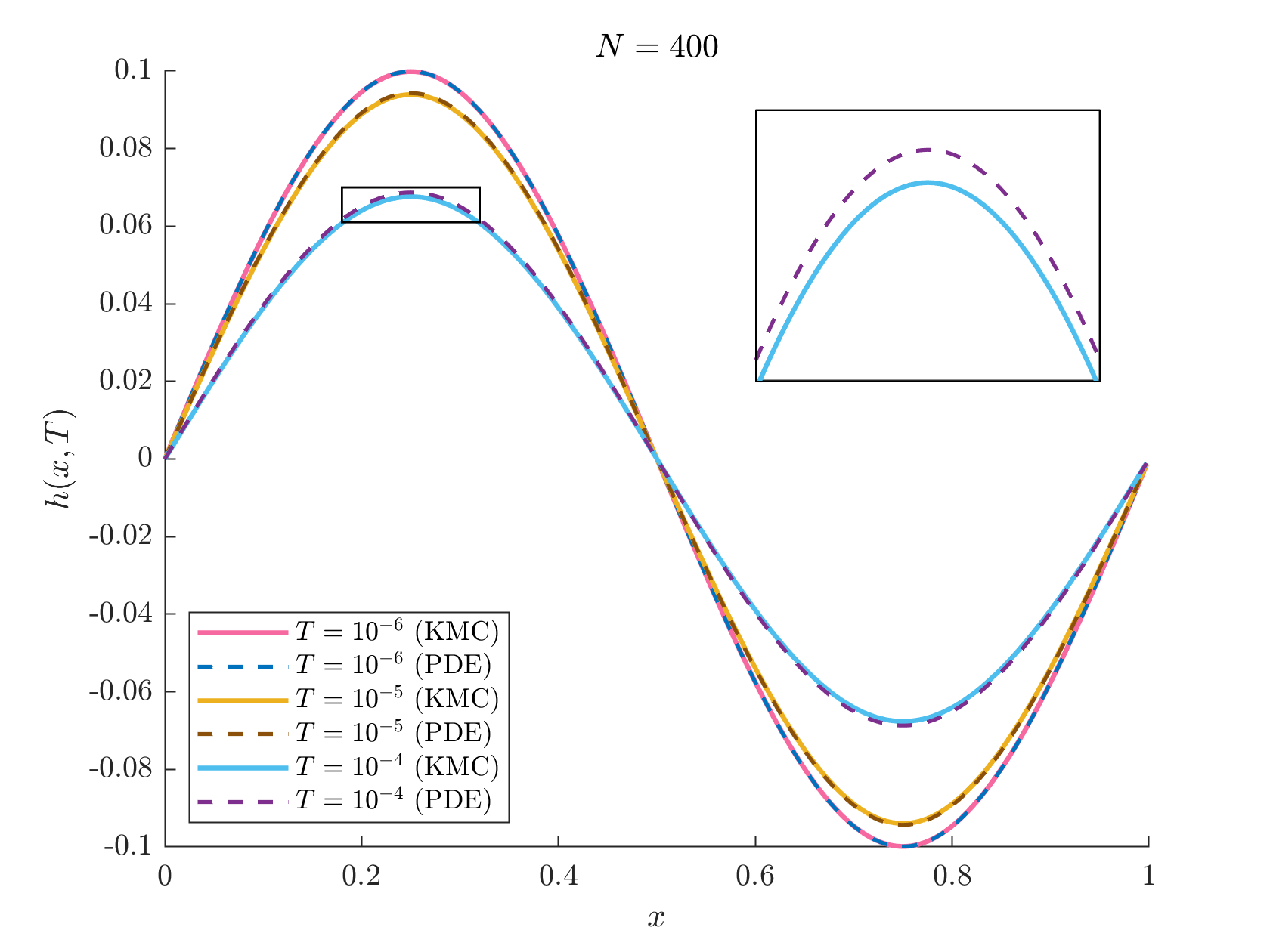}
\includegraphics[width = 0.495\linewidth]{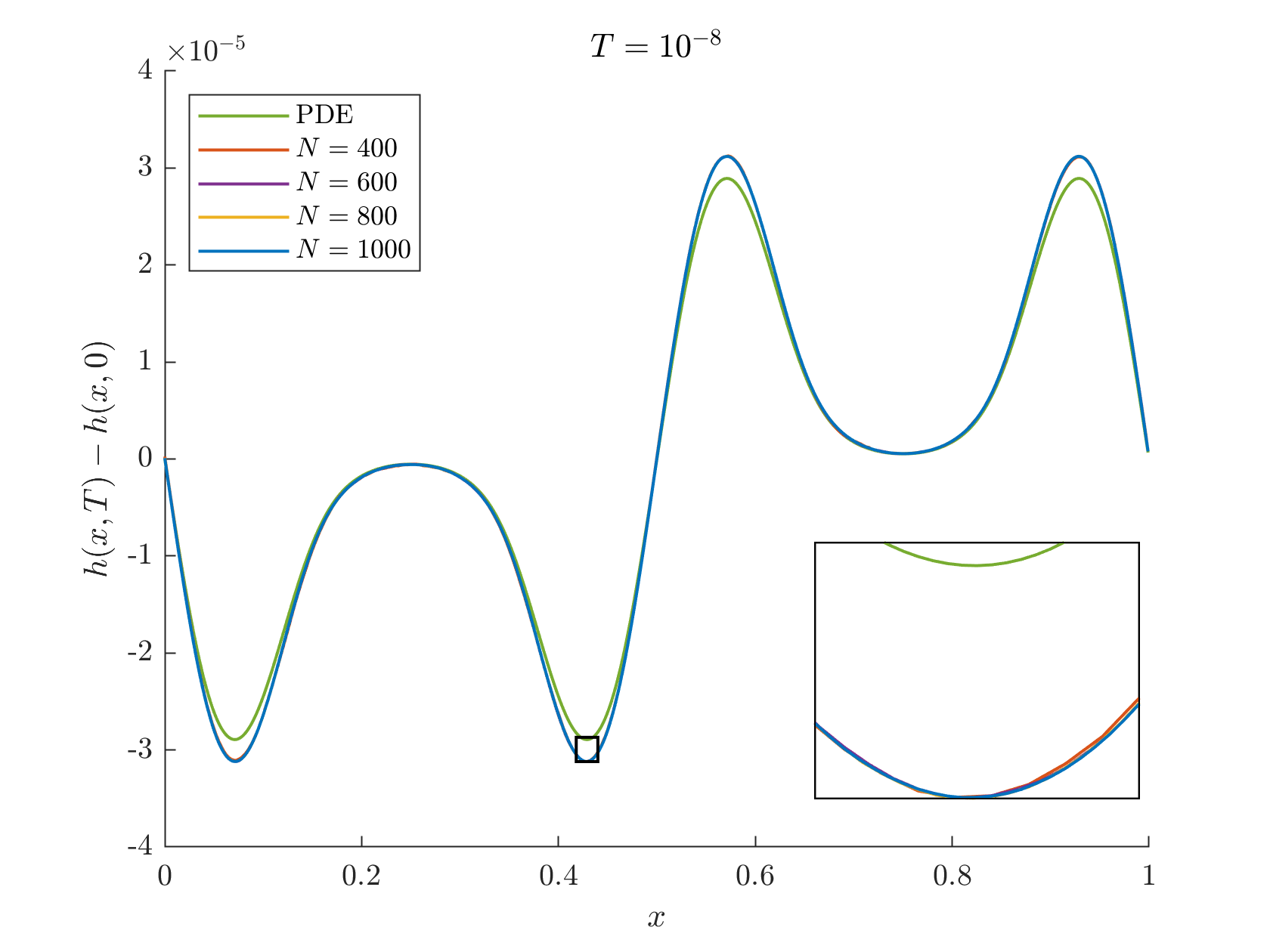}
\includegraphics[width = 0.495\linewidth]{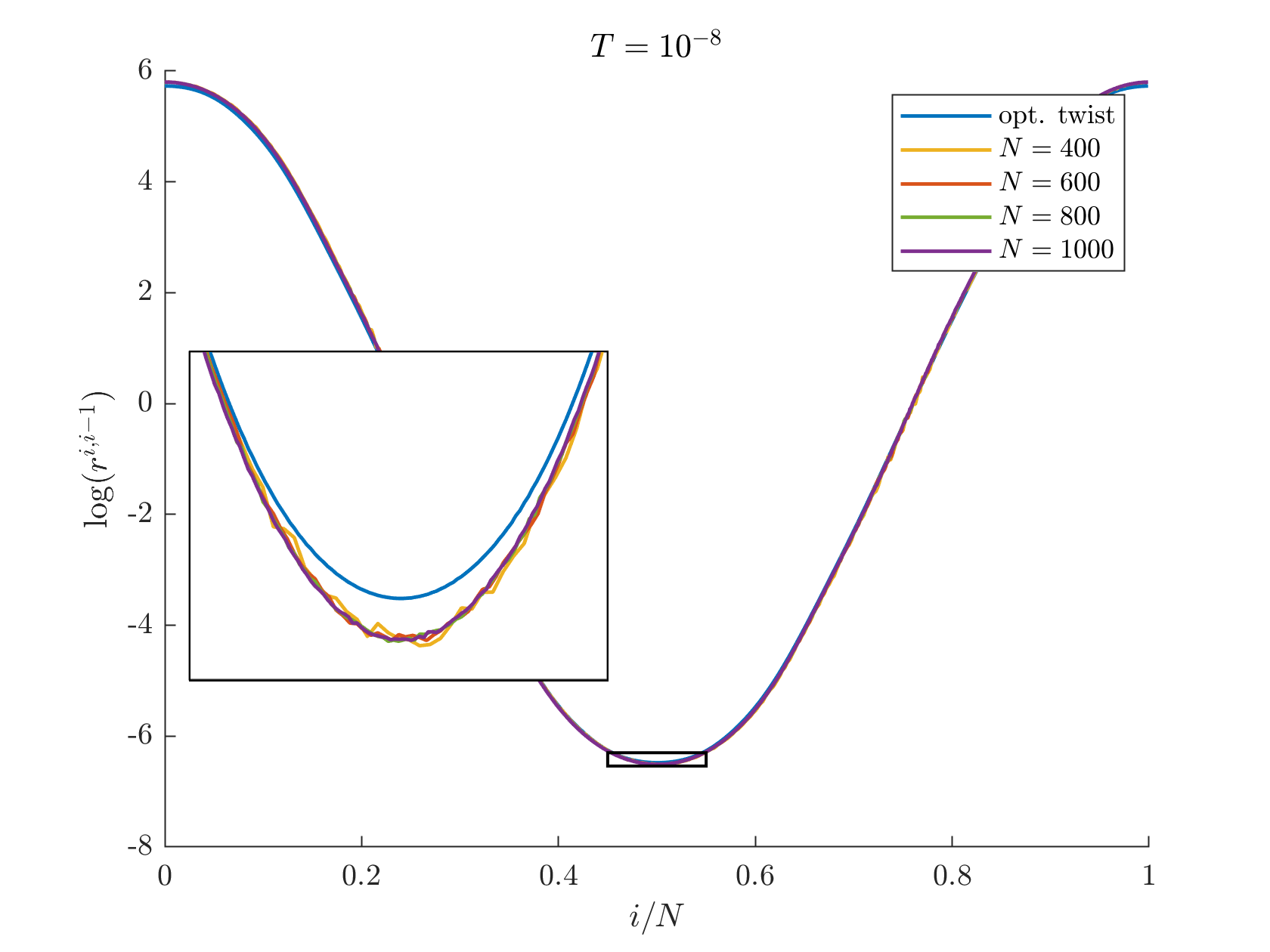}
\includegraphics[width = 0.495\linewidth]{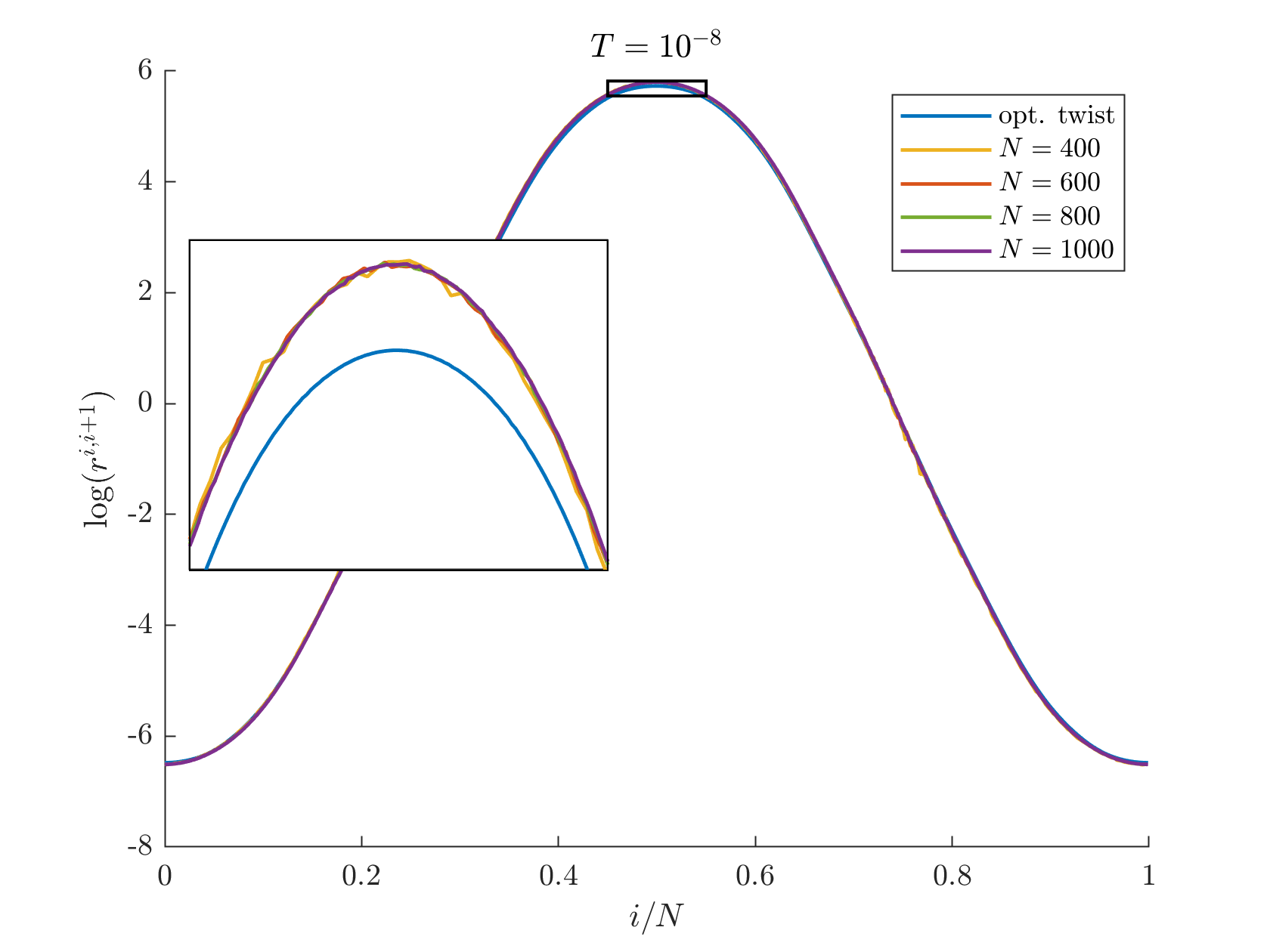}
\caption{Metropolis rate dynamics with temperature $\beta = 0.25$, from initial profile $h(x,0) = 0.1\sin(2\pi x)$. Top left: PDE vs KMC evolution with fixed $N$ at various times $T$. Top right: PDE vs KMC short time evolution, $h(x,T) - h(x,0)$, for fixed $T$ and various $N$. Bottom left (bottom right): optimal twist average vs empirical average of rates to jump left (jump right) for fixed $T$ and various $N$.}
\label{sine-met}
\end{figure}

Figure \ref{sine-met} shows results from the experiment with a sinusoidal initial profile and $\beta = 0.25$. Qualitatively, the PDE fully captures the microscopic dynamics. However, compared to the two-bump initial profile, we observed a slightly poorer fit between the PDE and microscopic dynamics. The top right and bottom figures, in which results from different values of $N$ are presented, show that the microscopic process
has nearly converged. Thus the discrepancy between the PDE and microscopic profiles is not attributable to $N$ being insufficiently large. Instead, it must be related to the assumption of a local Gibbs measure. Indeed, the bottom plots show that the expected value of the rate observable is not given exactly by the expectation with respect to the local Gibbs measure, though of course it is clearly a very close approximation to the actual measure. 

\vspace{0.2cm}
\noin\textbf{Discussion of Discrepancy in Figure~\ref{sine-met}.} The PDE~\eqref{anyaexpde} is derived using the assumption that once the microscopic process equilibrates locally, its probability distribution is well-approximated by a local Gibbs measure. In actuality, this measure requires a correction which involves the sum of the local currents, $\sum_{i=0}^{N-1}\J^i$. This sum (and the correction) vanishes for jump rates which satisfy the gradient condition, under which the current can be written as a discrete spatial derivative of another observable. As an example, symmetric jump rates (i.e. those for which jumping left from site $i$ has the same rate as jumping right) satisfy this condition. Indeed, we have $\J^i = r^{i} - r^{i+1},$ where $r^i := r^{i,i+1} = r^{i,i-1}$.

The correction to the measure affects the expectation of the current. It should manifest in the expectation as a nonconstant multiplicative factor, called a mobility by physicists. For a discussion of this, and for the form of the multiplicative factor, see e.g.~\cite{krug1995adatom, spohn2012large}. As we see in the numerical simulations, the correction has a negligible (albeit nonzero) effect on the dynamics. We expect the PDE~\eqref{anyaexpde} to be an accurate approximation to the microscopic dynamics when $\beta$ is small. Indeed, for $\beta$ small the jump rates are close to constant, so that $\sum_{i=0}^{N-1}\J^i$ is small.

For higher $\beta$ (lower temperature), the PDE requires a nontrivial correction. The form of this correction is investigated in a forthcoming paper. 


\section{Global solution and long time dynamics}
\label{sec:longtime}
In this section, we interpret the 4th order exponential PDE for Metropolis rate dynamics as a gradient flow of a proper convex functional with a $L^2$ dissipation. Then using the minimizing movement and the convergence analysis in \cite{AGS},  we  construct a global strong solution and prove there is no singularity formation.  In Appendix \ref{A:LTalt}, we present ideas for another approach using the bi-variational structures.

After setting most physical constants to be $1$, we obtain the continuous equation for surface growth with Metropolis-type rates 
\begin{equation}
\pd_t u = \frac{1}{2}\pd_x[ e^{-\pd_x^3 u}- e^{\pd_x^3 u} ].
\end{equation}
Denote $h:= \pd_x u$. We obtain formally the equation for $h$
\begin{equation}\label{PDE}
\pd_t h = \frac{1}{2}\pd_{xx}[e^{-\pd_{xx}h}-e^{\pd_{xx}h}]= -\pd_{xx}(\sinh (\pd_{xx}h)).
\end{equation}
We will validate this equation by  proving the global existence and the long time behavior of solutions to \eqref{PDE} with periodic boundary condition; see Theorem \ref{cor10}.

\subsection{Gradient flow in $L^2(\I)$}\label{sec1.2}
Let us first define formally a convex functional with some formal observations and recast \eqref{PDE} into a $L^2(\I)$ gradient flow. Let $\phi$ be
\begin{equation}
  \phi(h):= \int_\I \cosh(h_{xx}) \ud x.
\end{equation}
The first variation of $\phi$ is
$$\frac{\delta \phi}{\delta h}=\pd_{xx}(\sinh (h_{xx}))$$
and then formally we have
\begin{equation}\label{gfs}
  h_t=-\frac{\delta\phi}{\delta h}.
\end{equation}

To study the global strong solution to \eqref{PDE}, we plan to apply the gradient flow theory in metric space $L^2(\I)$.   Let us first make some inspiring observations, which will be made rigorous in the proof later.
\\
{\bf Observation 1 (Conservation Laws).} Thanks to the  periodic assumption, we have
\begin{equation}
  \frac{\ud}{\ud t} \int_\I h \ud x = 0,
\end{equation}
which implies $\int_\I h \ud x =\int_\I h_0 \ud x.$ Moreover from $\int_\I h_{xx} \ud x =0,$
we know
\begin{equation}
  \int_\I (h_{xx})^+ \ud x = \int_\I (h_{xx})^- \ud x.
\end{equation}
Here $(h_{xx})^- $ is the negative part of $h_{xx}$ and $(h_{xx})^+ $ is the positive part of $h_{xx}$. 
\\
{\bf Observation 2 (Dissipation Inequalities).}
From the gradient flow structure \eqref{gfs},
\begin{equation}
  \frac{\ud \phi}{\ud t} =\int_\I \frac{\delta\phi}{\delta h} h_t \ud x = -\int_\I \big|\frac{\delta\phi}{\delta h}\big|^2 \ud x =-\int_\I h_t^2 \ud x \leq 0,
\end{equation}
which gives the observation
$$\phi(h(t))\leq \phi(h(0)) \quad \text{ for any }t\geq 0.$$
One shall notice the boundedness of functional 
$\phi$  gives us good estimates to prevent both the positive and negative parts of $h_{xx}$ from becoming singular. 
Indeed, we have uniform estimate
\begin{align*}
  \frac{1}{p}\int_\I (h_{xx}^+)^p \ud x \leq & \int_{\I \cap (h_{xx}^+>0) } e^{(h_{xx})^+} \ud x \leq   \int_{\I \cap (h_{xx}^+>0) } \left( e^{(h_{xx})^+}+ e^{-(h_{xx})^+} \right) \ud x\\
  \leq  &\int_{\I  } \left( e^{h_{xx}}+ e^{-h_{xx}} \right) \ud x = 2\phi(h(t)) \leq 2\phi(h(0)).
\end{align*}
 Similarly, we have the same estimate for the negative  part $h_{xx}^-$.
Thus for any $p\in \mathbb{N}^+$,
\begin{equation}
\frac{1}{p}\int_\I |h_{xx}|^p \ud x \leq 4\phi(h(0)).
\end{equation}
For simplicity we choose the working space
\begin{equation}\label{Hnote}
H:=\left\{u\in L^2(\I); \int_\I u\ud x=0\right\},\quad
  V:= \{ u\in H^2(\I); \int_\I u\ud x=0\}
\end{equation}
with standard $L^2$-norm, denoted as $\|\cdot\|$, and $H^2$-norm, denoted as $\|\cdot\|_{V}$ respectively. Denote the best constant for Poincare's inequality as $\kappa>0$, which depends only on the size of the domain $\I$. Specifically, if the period of the domain is $L$, $\kappa=\left(\frac{L}{2\pi}\right)^2.$ 

\subsection{Variational inequality solution}
Let $\phi$ be a functional
\begin{equation}\label{phi}
 \phi:H\to [0,+\8],\qquad  \phi(h):= 
\begin{cases}
 \int_\I \cosh(h_{xx}) \ud x & \text{ if }h\in V;\\
 +\8 & \text{ otherwise. }
\end{cases} 
\end{equation}

\subsubsection{Euler Scheme}
First let us  establish the gradient flow evolution in the {metric}
space $(H,\mbox{dist})$, with distance $\mbox{dist}(u,v):=\|u-v\|$.
Let $h_0 (x)\in {H}$ be a given initial datum and $0<\tau\ll1$ be a given parameter.
We consider a sequence $\{x_n^{\tau}\}$ which satisfies the following unconditional-stable backward Euler scheme
\begin{equation}\label{E}
\left\{
\begin{array}{l}
x^{(\tau)}_n\in \text{argmin}_{x'\in H} \left\{\phi (x')+\dfrac1{2\tau} \|x'-x^{(\tau)}_{n-1}\|^2 
\right\},  \qquad n\ge 1,\\
x^{(\tau)}_0:= h_0 \in H.
\end{array}
\right.
\end{equation}
The existence and uniqueness of the sequence $\{x_n^{\tau}\}$ can be proved by direct methods in the calculus of variations after establishing the convexity and lower semicontinuity of $\phi$ in Proposition \ref{convex}.
Thus we consider the gradient descent with respect to $\phi$ in the space $(H,\mbox{dist})$.

Now for any $0<\tau\ll1$ we define the resolvent operator, also known as proximal mapping of $\phi$,  (see \cite[p. 40]{AGS})
\begin{equation*}
\mathcal{J}_\tau[h]:=\text{argmin}_{v\in H} \left\{\phi (v)+\dfrac1{2\tau} \|v-h\|^2 \right\},
\end{equation*}
then the variational approximation of $h$ at $t$ is obtained by Euler scheme \eqref{E} as
\begin{equation}\label{Euler10}
h_n(t):=(\mathcal{J}_{t/n})^n[h_0 ].
\end{equation}
In Proposition \ref{EVI}, we will use the theory for gradient flow in metric space \cite[Theorem~4.0.4]{AGS} to establish the convergence of the variational approximation $h_n(t)$ to variational inequality solution to \eqref{PDE}, which is defined below.
\begin{definition}\label{defweak}
Given initial data $h_0 \in H$, we call $h:[0,+\8)\to H$ a variational inequality solution to \eqref{PDE} if $h(t)$ is a locally absolutely continuous curve such that $\lim_{t\to 0} h(t)=h_0 $ in $H$ and
 \begin{equation}
\la h_t(t),h(t)-v\ra_{{ H',H}}\le \phi(v)-\phi(h(t)) \quad \text{for a.e. } t>0,\,\forall v\in { D(\phi)}.
\end{equation}
\end{definition}
Next we study some properties, including convexity and lower semicontinuity in $H$, of the functional $\phi $.

\subsection{Convexity and lower semicontinuity of function $\phi $ in $H$}

We will prove the $\lambda$-convexity and lower semicontinuity of function $\phi$ in $H$. We note $\lambda>0$ is important for the long time behavior of the global solution.
\begin{proposition}\label{convex}
The functional $\phi :H\to [0,+\8]$ is proper, $\lambda$-convex, lower semicontinuous in $H$ and satisfies coercivity defined in \cite[(2.4.10)]{AGS}.
\end{proposition}

\begin{proof}
Clearly since the typical function $h= 0\in D(\phi )$, so $D(\phi )=\{\phi <+\8\}$ is nonempty and $\phi $ is proper.
Due to the positivity of $\phi$, coercivity \cite[(2.4.10)]{AGS}, {i.e.,
$\exists u*\in D(\phi ), r*>0 \text{ such that } \inf\{\phi  (v): v\in H, \text{dist}(v,u*)\leq r*\}>-\infty,$
} is obvious.
\medskip

{\em $\lambda$-Convexity.} 
Given $u,v\in H$, $t\in (0,1)$, without loss of generality we assume $u,v\in D(\phi )$,
otherwise the convexity inequality is trivial. Therefore
from the definition of $\lambda$-convexity, we only need to prove for any $t\in[0,1]$, any $u,v\in H$ we have
\begin{equation}\label{lam-c}
\phi((1-t)u+t v)\leq (1-t)\phi(u)+ t \phi(v)-\frac12\lambda t(1-t) \|u-v\|_{L^2}^2.
\end{equation}
Denote $$I(t):= \int_{\I} (1-t)\cosh(u_{xx})+ t \cosh(v_{xx})- \frac{\lambda}{2}t(1-t)\|u-v\|_{L^2}^2-\cosh((1-t)u_{xx}+tv_{xx})\ud x$$
and notice $I(1)=I(0)=0$. Thus we only need to prove $I''(t)\leq 0$. It is easy to calculate that
$$I''(t)=\int_{\I} -\cosh [(1-t)u_{xx}+t v_{xx}](v_{xx}-u_{xx})^2 + \lambda(u-v)^2 \ud x\leq \int_{\I}-(u_{xx}-v_{xx})^2+ \lambda(u-v)^2 \ud x \leq 0 $$
due to Poincare's inequality with $\lambda=\frac{1}{\kappa^2}$.
Hence $\phi $ is $\lambda$-convex for $\lambda=\frac{1}{\kappa^2}=\left(\frac{2\pi}{L}\right)^4>0$.

\medskip

{\em Lower semicontinuity.} Consider a sequence $h_n\to h$ in $H$. We need to check
$$\phi (h)\le \liminf_n \phi (h_n).$$
If $h_n\in D(\phi )$ does not hold
for all large $n$, then lower semicontinuity holds.
Without loss of generality, we can assume $h_n\in D(\phi )$ for all $n$, and also
$$\liminf_n \phi (h_n)=\lim_n \phi (h_n).$$

First notice $h_n\in D(\phi)$ for any $n$ implies uniform estimate
\begin{align*}
  \frac{1}{p}\int_\I |(\partial_{xx}h_{n})^+|^p \ud x \leq & \int_{\I \cap (\partial_{xx}h_{n}^+>0) } e^{(\partial_{xx}h_{n})^+} \ud x \leq   \int_{\I \cap (\partial_{xx}h_{n}^+>0) } \left( e^{(\partial_{xx}h_{n})^+}+ e^{-(\partial_{xx}h_{n})^+} \right) \ud x\\
  \leq  &\int_{\I  } \left( e^{\partial_{xx}h_{n}}+ e^{-\partial_{xx}h_{n}} \right) \ud x = 2\phi(h_n) < \8.
\end{align*}
 Similarly, we have the same estimate for the negative  part $h_{xx}^-$.
Thus for  $p=2$,
\begin{equation}
\int_\I |h_{xx}|^2 \ud x \leq C,
\end{equation}
which yields that there exists $h^*\in V$ such that $h_n \rightharpoonup h^*$ in $V$. From the strong convergence $h_n \to h$ in $H$ we know the 
$h_n \rightharpoonup h$ in $V$.
Therefore from the convexity of $\cosh$ function, we know $\phi$ is also convex in $V$ and  lower semicontinuous w.r.t the weak topology of $V$
\begin{equation*}
\liminf_n \phi(h_n) \geq \phi(h).
\end{equation*}
Thus the lower semicontinuity in $H$ is proved.
\end{proof}

As long as we have the convexity of $\phi $, the $(\tau^{-1}+ \lambda)$-convexity is standard and the proof can be found in \cite[Section 2.4]{AGS}.
\begin{proposition}[$(\tau^{-1}+\lambda)$-convexity]\label{convex2}
For any $h,v_0,v_1\in D(\phi )$, there exists a curve $v:[0,1]\to D(\phi )$
such that $v(0)=v_0,\, v(1)=v_1$ and
the functional
\begin{equation}\label{PhiD}
\Phi(\tau,h;v):=\phi (v)+\frac1{2\tau} \|h-v\|_H^2
\end{equation}
satisfies $\tau^{-1}$-convexity, i.e.,
\begin{equation}\label{c}
\Phi(\tau,h;v(t))\le (1-t)\Phi(\tau,h;v_0)+t\Phi(\tau,h;v_1)-\frac1{2}(1/\tau+\lambda) t(1-t)\|v_0-v_1\|_H^2
\end{equation}
for all $\tau>0$, $t\in[0,1]$.
\end{proposition}

\subsection{Existence and Long time behavior of the global solution}\label{Sec2.4}
After studying convexity and lower semicontinuity in the last section, we shall apply the convergence result in \cite[Theorem 4.0.4]{AGS} to derive that the discrete solution $h_n$ obtained by Euler scheme \eqref{E} converges to the variational inequality solution defined in Definition \ref{defweak}. For $v\in D(\phi)$, denote the local slope
\begin{equation}\label{localslope}
 |\pd \phi|(v):=\limsup_{w\to v}\frac{\max\{\phi(v)-\phi(w),0\}}{\mbox{dist}(v,w)}.
\end{equation}
\begin{remark}
In particular, by \cite[Proposition 1.4.4]{AGS}, we have the local slope is
$$|\pd \phi|(v)= \min \{\|\xi\|; \xi\in \pd \phi(v)\}$$
for any $v\in H$. Since $\phi$ is a smooth functional, its subdifferential $\pd \phi$ is single-valued and equals its Fr\`echet differential 
$$\pd \phi =D\phi(h):=[\sinh(h_{xx})]_{xx}.$$
\end{remark}
\begin{proposition}\label{EVI}
  Given $h_0  \in H$, for any $t>0$, $t=n\tau$, let $h_n(t)$ defined in \eqref{Euler10} be the approximation solution obtained by Euler scheme \eqref{E}, then
\begin{enumerate}
\item There exists a local Lipschitz curve $h(t):  [0,+ \8) \to H \in MM(\Phi; h_0)$ (i.e. minimizing movement for $\Phi$) such that
      \begin{equation}\label{tmp23}
        h_n(t)\to h(t)\text{  in }L^2(\I)
      \end{equation}
      and $h:[0,+\8)\to H$ is the unique EVI solution in the sense that $h$ is unique among all the locally absolutely continuous curves such that $\lim_{t\to 0} h(t)=h_0 $ in $H$ and
\begin{equation}\label{vi}
\frac12\frac\ud{\ud t}\|h(t)-v\|^2+ \frac{\lambda}{2}\|h(t)-v\|^2\le \phi (v)-\phi (h(t)), \quad \text{  a.e. } t>0,\,\forall v\in D(\phi );
\end{equation}
\item 
We have the following regularities
\begin{align}
\phi (h(t)) &\le \phi (v)+\frac1{2t}\|v-h_0 \|_H^2, \qquad \forall v\in D(\phi ),\label{phi-dec}\\
|\pd\phi |^2(h(t)) &\le |\pd\phi |^2(v)+\frac1{t^2}\|v-h_0 \|_H^2,\qquad\forall v\in D(|\pd \phi |);\label{Dphi-dec}
\end{align}
\item 
There exist $t_0>0$ and we have the exponential decay of $h(t)$
\begin{equation}
\frac{\lambda}{2}\|h(t)-h^*\|\leq (\phi(h_0)-\phi(h^*))e^{-2\lambda(t-t_0)}, \quad \text{ for any }t\geq t_0>0,
\end{equation}
where $h^*=0$ is the unique minimizer of $\phi$.
\end{enumerate}
\end{proposition}
This Proposition is a direct result by combining \cite[Theorem 4.0.4]{AGS} and \cite[Theorem 2.4.14]{AGS} with Proposition \ref{convex} and Proposition \ref{convex2}.  We remark the exponential convergence rate $\lambda=\frac{1}{\kappa^2}=\left(\frac{2\pi}{L} \right)^4$. So the convergence speed is the square of  the Poincare-Wirtinger's constant.

Next by Proposition \ref{convex} and \cite[Theorem 2.4.15]{AGS}, we claim that given a better initial data $h_0$, the EVI solution obtained above is a global strong solution to \eqref{PDE} with better properties as follows.
\begin{theorem}\label{cor10}
Given any  $T>0$ and initial datum $h_0  \in D(\pd \phi)$ such that $\phi(h_0 )<+\8$, the solution obtained in Proposition \ref{EVI} is a global strong solution
in the sense that $\pd_{t} h = -\pd \phi= -\pd_{xx}(\sinh(h_{xx}))$ holds for all $t\geq 0$
with the following regularities
 $$h\in C([0,T];D(\pd \phi) )\cap C^1([0,T];H)$$
and the decay estimate
\begin{equation}\label{decay1}
 \|h(t)\|\leq \|h_0\|, \quad \forall t\geq 0,\quad
 \|\pd_t h(t)\|=\|\pd \phi (h(t))\|\leq e^{-\lambda t}\|\pd \phi (h_0)\|_{L^2}, \quad \forall t\geq 0.
\end{equation}
\end{theorem}
\begin{remark}
In general $D(\pd \phi)$ is given by all $v\in D(\phi)$ such that $\partial \phi(v) \neq \emptyset.$ In our case, $D(\pd \phi)=\{h\in V; (\sinh(h_{xx}))_{xx}\in H\}$.
\end{remark}
\begin{proof}[Proof of Theorem \ref{cor10}]
Since Proposition \ref{convex} verifies all the assumptions in \cite[Theorem 2.4.15]{AGS}, so we directly apply \cite[Theorem 2.4.15]{AGS} for initial datum $h_0  \in D(\pd \phi)$ (see \cite[Remark 2.4.16]{AGS}). First $\pd_{t} h = -\pd \phi= -\pd_{xx}(\sinh(h_{xx}))$ follows from \cite[Theorem 2.3.3]{AGS} and \cite[Corollary1.4.2]{AGS}.  Then the statement (i) and (iii) in \cite[Theorem 2.4.15]{AGS} show that $h\in C([0,T]; D(\pd \phi))$ and
\begin{equation}
    |\pd_{t^+} h|(t):=\lim_{s\to t^+} \frac{\|h(s)-h(t)\|}{s-t}=|\pd \phi|(h(t))=\|\sinh(h_{xx})_{xx}\|  \quad \text{for any } t\geq 0.
\end{equation}
Finally the non-increasing property of $\|h(t)\|$ follows from \eqref{vi} while the non-increasing property of $e^{\lambda t}\|\pd \phi (h(t))\|$ follows from the statement (ii) in \cite[Theorem 2.4.15]{AGS}.
\end{proof}

\section{Numerics} \label{sec:numerics}

In this section, we numerically explore some properties of the PDE \eqref{anyaexpde}. Due to the exponential dependence on $\beta\partial_{xxx}h$ in \eqref{anyaexpde}, we are limited to the setting of either relatively high temperature (small $\beta$) or small curvature, in which case the PDE can be solved numerically. As in \cite{mw-krug}, we focus on two key phenomena: $(a)$ wetting, or how compactly supported solutions evolve to fill the domain and $(b)$ self-similar structures in the collapse to equilibrium.  Similar numerical studies were undertaken in \cite{mw-krug} for the Arrhenius rate PDE \eqref{Hm1expde}.  While certainly not an exhaustive study of phenomena in these models, they are key features of the dynamics one would like to understand for the evolution of crystal surfaces.  The numerical method implemented below involved using a symmetric finite difference stencil to discretize spatial derivatives, along with a stiff ODE solver such as {\it ode15s} in {\it Matlab} in time.  Similar methods were employed in \cite{mw-krug,liu2017asymmetry} and are described in some more detail there.

\subsection{Wetting}
Motivated by properties of \eqref{Hm1expde} shown in \cite{mw-krug}, one phenomenon we investigate for \eqref{anyaexpde} is how quickly mass spreads from regions of non-zero height into regions with zero height.  This process is known as wetting in the study of thin films. In order for facets (macroscopic flat regions on the crystal surface) to be stable features of a surface, the wetting rate should be finite. In Figure \ref{fig:wetting1}, we study this phenomenon for the PDE \eqref{anyaexpde}. 
Similar to \eqref{Hm1expde}, it appears numerically that the solution can wet at finite rate.
One interesting difference between \eqref{anyaexpde} and \eqref{Hm1expde} is that for an initial nonnegative compactly supported profile, the numerical solution to \eqref{Hm1expde} remains positive while the numerical solution to \eqref{anyaexpde} dips below zero before levelling off. The wetting rate was investigated for the initial profile
\begin{equation}
\label{eqn:initdata2}
h(0,x) = \left\{ \begin{array}{cc}
 e^{8-|x|^{-1} - (0.5-|x|)^{-1}} \ \  &\text{for} \ 0 < |x| < \frac12, \\
0 \ \ &\text{otherwise}.
\end{array} \right.
\end{equation}

Given the high curvatures present in this model, the discretization in space is taken to be somewhat coarse when using the discrete approximation we have implemented here to the flow of \eqref{anyaexpde} given by using symmetric finite difference operators in space and the variable time step stiff solver {\it ode15s} in {\it Matlab} in time. Our methods are comparable to the methods used in \cite{mw-krug,liu2017asymmetry}.  The results are consistent as we refine the spatial grid to the extent possible, but the grid can only be refined so much for these extremely stiff approximations using the PDE solvers implemented here due to the high curvatures present.  For related algorithms that can be implemented with much finer grids but that require longer running times, see the recent works of \cite{liu2016existence,craig2020proximal}, for which similar and consistent dynamics have been observed for PDEs of the form \eqref{Hm1expde}.  

\begin{figure}
\centering
\includegraphics[width=2.5in]{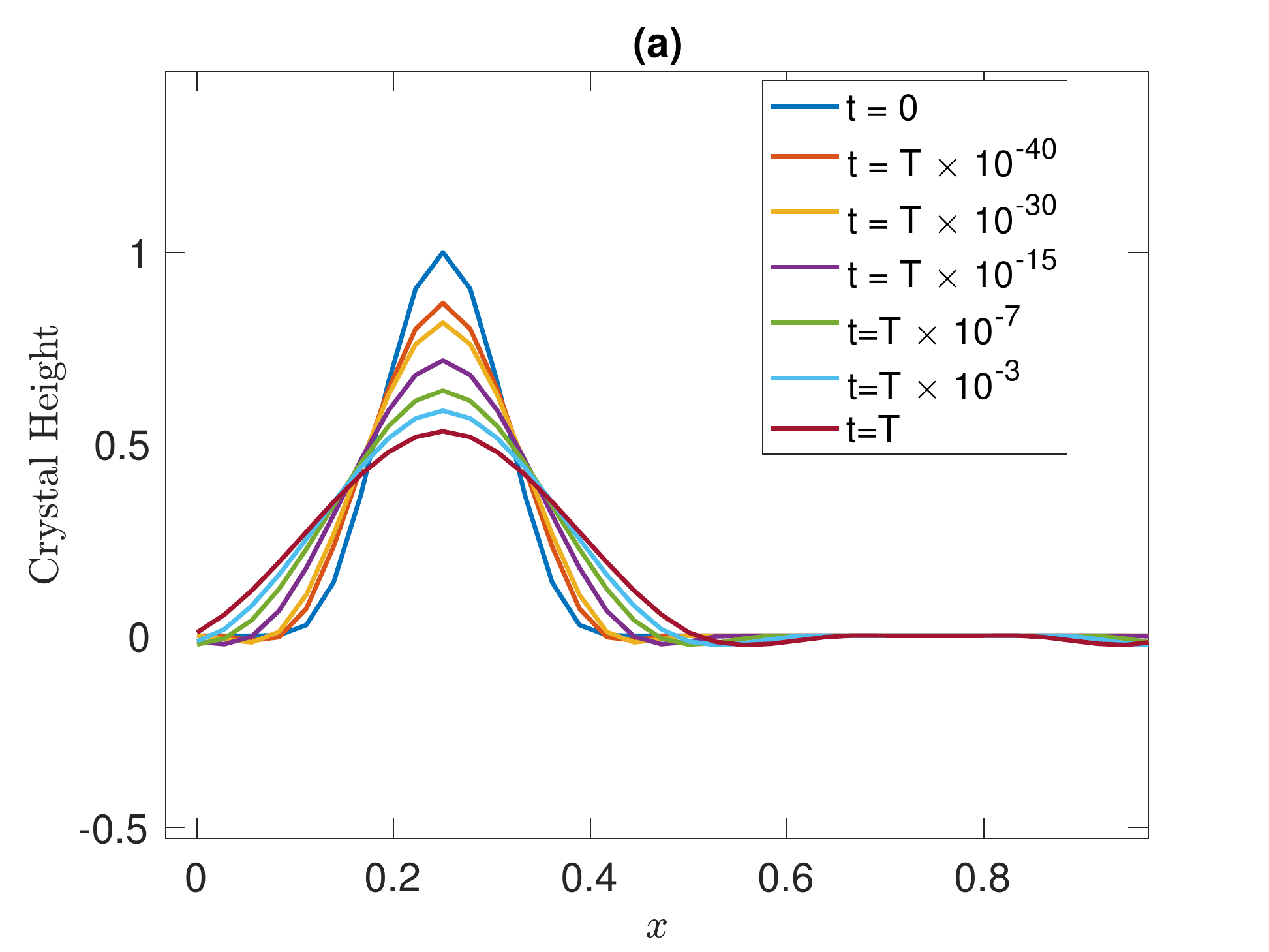}
\includegraphics[width=2.5in]{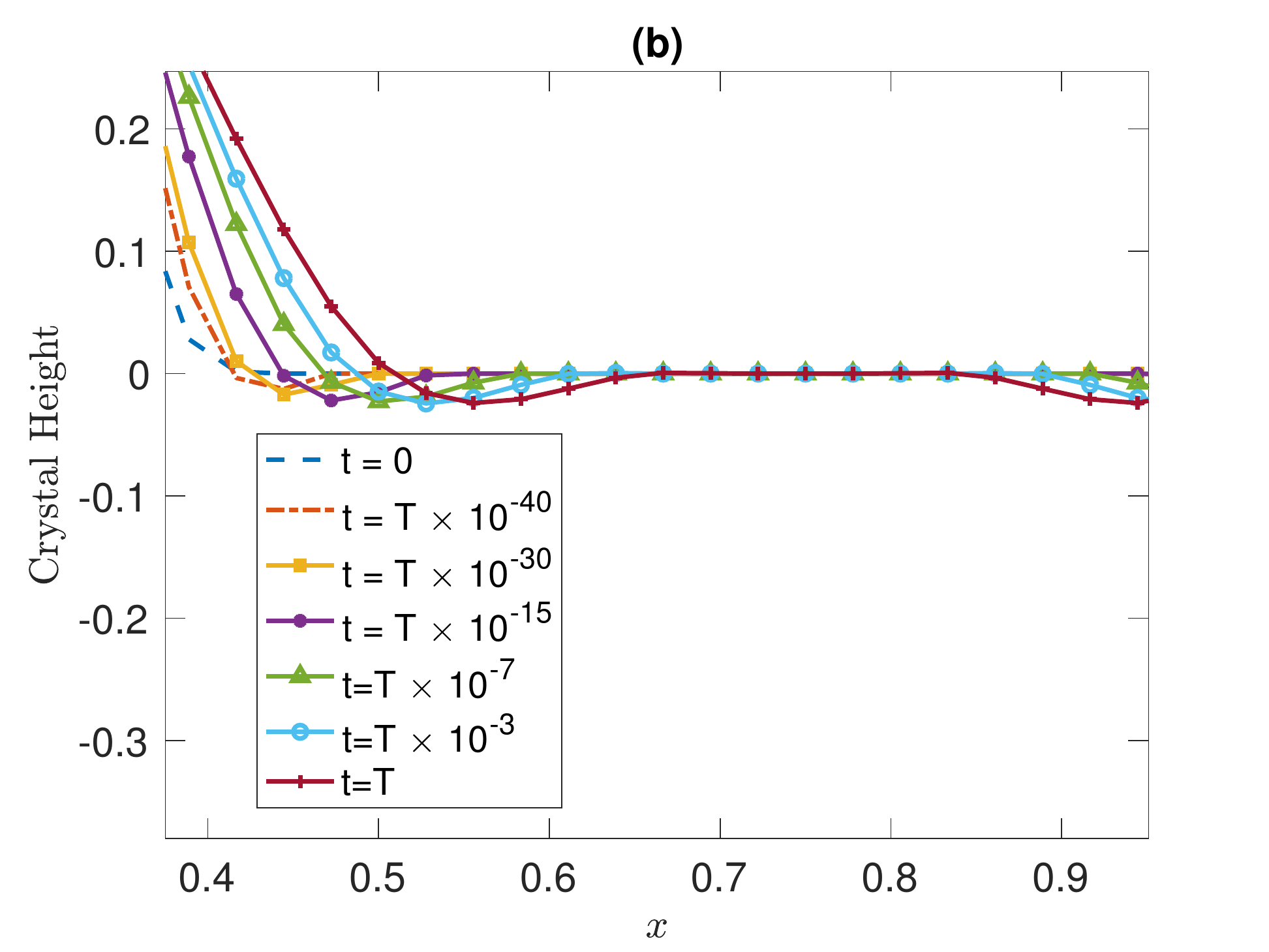}
\caption{ (Color online) Snapshots of solution of PDE \eqref{anyaexpde} with $\beta=.05,$  from the initial profile in \eqref{eqn:initdata2}, at various times in an interval of length $T=5\times 10^{-10}$ $(a)$ and a blowup in the region of zero initial height $(b)$.}
\label{fig:wetting1}
\end{figure}

\subsection{Self-Similarity}

Again following the analysis in \cite{mw-krug}, we study the behavior of the surfaces as they near equilibrium ($h\equiv 0$).  In Figure \ref{fig:selfsimilar} we show that the surfaces appear to approximately factor as $h(t,x) = \phi(t) g(x)$ for very large $t.$  The results in that figure are generated via a fixed point iteration in which the surface is evolved for some length of time and then rescaled so that the surface's maximal (in absolute value) height is 1, and then evolved and rescaled repeatedly until convergence. The plot shows the last two fixed point iterations (before rescaling).  The fact that they nearly coincide indicates that the iterations have converged to $g(x)$. We note that the function $g(x)$ will typically have some dependence on the particular initial profile. In this simulation we took $h(0,x) = \sin(2\pi x)$ and $\beta=.25$. It appears that the self-similar solution is quite regular. This is in contrast to \eqref{Hm1expde}, in which a singularity forms in the self-similar profile at its minimum.  
\begin{figure} 
\centering
\includegraphics[width=3.5in]{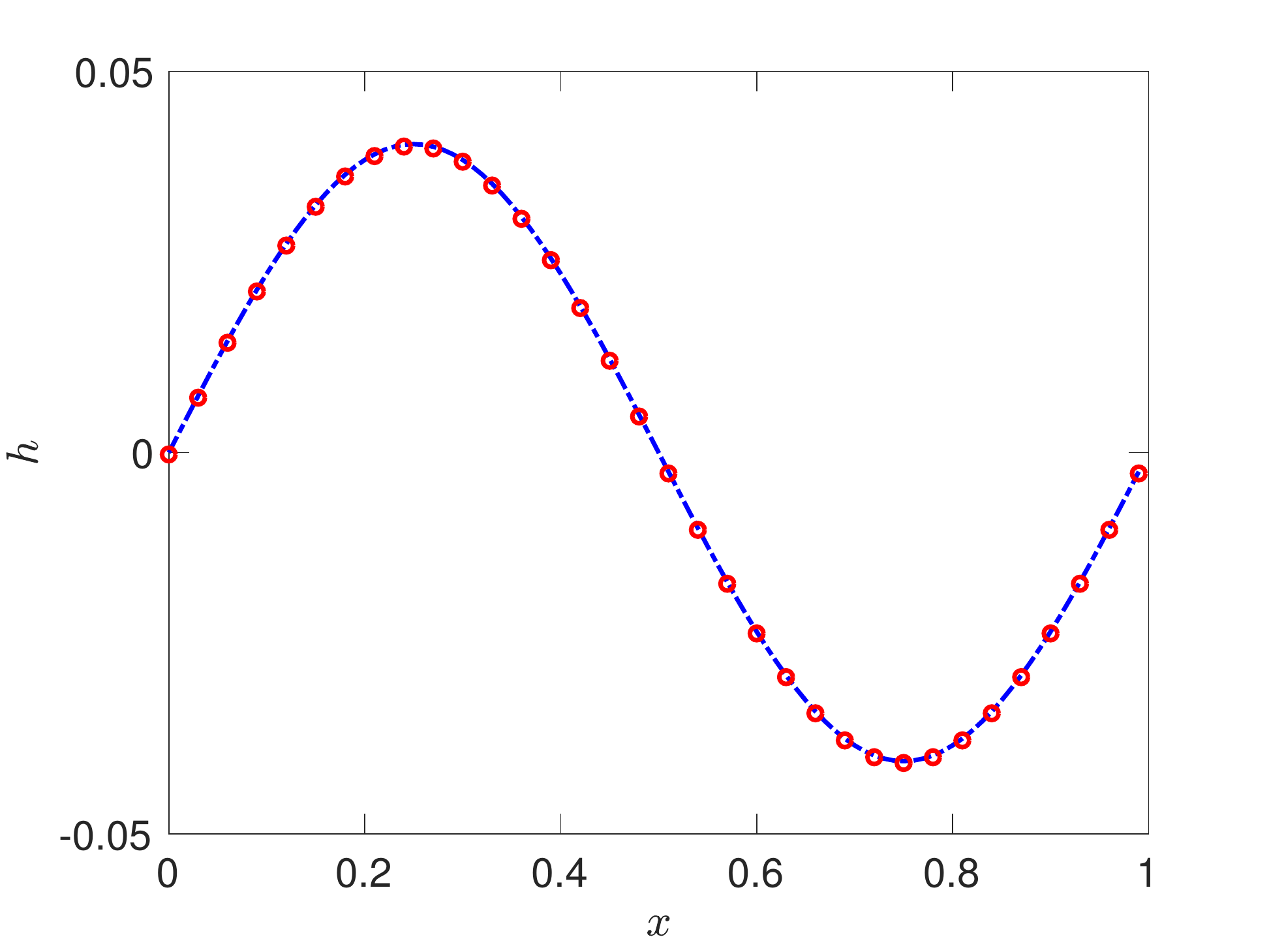}
\caption{ (Color online) Results of fixed-point iteration in which PDE  \eqref{anyaexpde} is evolved for some interval of time, then rescaled to have maximum height (in absolute value) equal to 1 and then evolved and rescaled repeatedly until convergence. The plot shows the last two fixed point iterations (before rescaling). The solution appears to be approximately of the form $h(t,x) = g(x) \phi (t).$ Here, we renormalize after intervals of length $T=5e-4$.}
\label{fig:selfsimilar}
\end{figure}

\appendix
\section{Expectation of Jump Rates}\label{rate-expect}
Recall the definition $z_N^i = h_N^{i+1}-h_N^i$, and the expression for the rates:
\beq\label{metrop-def} r^{i,j}(h) = \e\left(-\frac{\beta}{2}\bigg[H(J_i^jh)-H(h)\bigg]\right),\quad |i-j|=1,\eeq where 
$$H(h) = \sum_{i=0}^{N-1}(z^i)^2.$$ 
A simple computation gives the explicit expression
\beq\label{met-r_right}r^{i,i+1}(z_N) =e^{-3\beta} \text{exp}\left(-\beta(z^{i-1}-2z^i + z^{i+1})\right)\eeq and
\beq\label{met-r_left}r^{i+1,i}(z_N) =e^{-3\beta} \text{exp}\left(\beta(z^{i-1}-2z^i + z^{i+1})\right).\eeq

We will compute the expectation $\langle\J^i\rangle = \langle r^{i,i+1}\rangle - \langle r^{i+1,i}\rangle$ with respect to the local Gibbs measure, assumed to be a good approximation to the true measure at macroscopic time $t>0$ (microscopic time $N^4t$) and for large enough $N$. 
This measure is given by
$$p_{\bm\lambda}(z_N) \propto \e\left(-\beta\sum_{i=0}^{N-1}(z^i)^2 + \sum_{i=0}^{N-1}\lambda^iz^i\right),\quad z_N\in\Z^N.$$ The $\lambda^i$ are chosen so that under $\rho_{\lambda}$, the expectation of $z_N^i$ is $N^2h_x(t, i/N)$. One can show that this implies $N^{-2}\lambda^{Nx}\to 2\beta h_x(t,x)$ as $N\to\infty.$

Note that $p_{\bm\lambda}$ is a product of one dimensional measures which we denote $p_{\lambda_i}$, where $p_{\lambda}(n) \propto e^{-\beta n^2 + \lambda n}.$ We let $\langle\cdot\rangle_{\bm\lambda}$ denote expectation with respect to $p_\lambda$ and $\langle\cdot\rangle_{\lambda_i}$ denote expectation with respect to $p_{\lambda_i}$. Using the expression~\eqref{met-r_right}, we have
\beq\label{rate-prod}\langle r^{i,i+1}(z_N)\rangle_{\bm\lambda}= e^{-3\beta}\langle e^{\beta z}\rangle_{\lambda_{i-1}}\langle e^{-2\beta z}\rangle_{\lambda_{i}}\langle e^{\beta z}\rangle_{\lambda_{i+1}}.\eeq
\begin{lemma}\label{exponential_expectation}
	Let $m\in\mathbb Z$ and $z\in\mathbb Z$ be a random variable distributed according to $p_{\lambda}$, where $p_{\lambda}(n) \propto e^{-\beta n^2 + \lambda n}.$ Then 
	\begin{equation}\label{exp-exp}
	\langle e^{m\beta z}\rangle_{\lambda} =\begin{cases}
	e^{\beta m^2/4 + \lambda m/2},\quad\text{if }m\text{ even,}
	\\ Z\left(\beta, \frac{\lambda}{2\beta}\right)e^{\beta m^2/4 + \lambda m/2}\quad\text{if }m\text{ odd,}
	\end{cases}
	\end{equation} where 
	$$Z(\beta,\alpha) = \frac{\sum_{n=-\infty}^{\infty}e^{-\beta(n-(\alpha + \frac12))^2}}{\sum_{n=-\infty}^{\infty}e^{-\beta(n-\alpha)^2}}.$$ If $\beta$ is small, then $Z(\beta, \alpha) = 1 + o(\beta)$. 
\end{lemma}
\begin{proof} We have
\begin{equation}\begin{split}
	\langle e^{m\beta z}\rangle_{\lambda} &= \frac{\sum_{n=-\infty}^{\infty}e^{m\beta n-\beta n^2 + \lambda n }}{\sum_{n=-\infty}^{\infty}e^{-\beta n^2 + \lambda n}} \\
	&= \text{exp}\left(\beta\left(\frac m2 + \frac{\lambda}{2\beta}\right)\right)^2 - (\frac{\lambda}{2\beta})^2)\frac{\sum_{n=-\infty}^{\infty}e^{-\beta(n - (\frac m2 + \frac{\lambda}{2\beta}))^2 }}{\sum_{n=-\infty}^{\infty}e^{-\beta (n - \frac{\lambda}{2\beta})^2}}.
\end{split}\end{equation}
The factor in front of the ratio of sums simplifies to 
$$	e^{\beta m^2/4 + \lambda m/2}.$$ If $m$ is even then, by summing over $n-\frac m2$ , we see that the numerator of the sum ratio equals the denominator. If $n$ is odd, we can sum over $n-\frac{m-1}{2}$ in the numerator to obtain
$$\frac{\sum_{n=-\infty}^{\infty}e^{-\beta(n - \frac12 - \frac{\lambda}{2\beta})^2 }}{\sum_{n=-\infty}^{\infty}e^{-\beta (n - \frac{\lambda}{2\beta})^2}}.$$ To see that $Z(\beta, \alpha) = 1 +o(\beta)$ one can express the sums in terms of the Jacobi theta function $\vartheta_3$ and use properties of this function. 
\end{proof}
Using Lemma \ref{exponential_expectation}, we can finish computing the expectation $\langle r^{i,i+1}\rangle_{\bm\lambda}$. Substituting~\eqref{exp-exp} into~\eqref{rate-prod}, we obtain

\beqs
\langle r^{i,i+1}\rangle_{\bm\lambda} &= e^{-3\beta}\left(e^{\beta/4 + \lambda^{i-1}/2}\right)\left(e^{\beta - \lambda^i}\right)\left(e^{\beta/4 + \lambda^{i+1}/2}\right)Z(\lambda^{i-1},\beta)Z(\lambda^{i+1},\beta) \\
&=e^{-\frac32\beta}\mathrm{exp}\left(\frac12\left(\lambda^{i-1} - 2\lambda^{i} + \lambda^{i+1}\right)\right)Z(\lambda^{i-1},\beta)Z(\lambda^{i+1},\beta)\\
&\approx e^{-\frac32\beta}\mathrm{exp}\left(\frac12\left(\lambda^{i-1} - 2\lambda^{i} + \lambda^{i+1}\right)\right),
\eeqs
where the last line is for $\beta$ small. Similarly, we have
$$\langle r^{i+1,i}\rangle_{\bm\lambda} \approx e^{-\frac32\beta}\mathrm{exp}\left(-\frac12\left(\lambda^{i-1} - 2\lambda^{i} + \lambda^{i+1}\right)\right).$$

Since $N^{-2}\lambda^i\approx 2\beta h_x(t,i/N)$ for $N$ large, we have
$$\lambda^{i-1} - 2\lambda^{i} + \lambda^{i+1}\approx 2\beta h_{xxx}(t,i/N),$$ so that
$$\langle\J^i\rangle_{\bm\lambda}=\langle r^{i,i+1}-r^{i+1,i}\rangle_{\bm\lambda}\approx e^{-\frac32\beta}\sinh(\beta h_{xxx}(t,i/N)).$$

\section{Another method for long time behavior using bi-variational structures}
\label{A:LTalt}
Since $\phi$ is a smooth functional, its subdifferential $\pd \phi$ is single-valued and equals its Fr\'echet differential $\pd \phi(h)=[\sinh(h_{xx})]_{xx}.$ We give another simple proof for the exponential decay to $0$ of the global classical solution $h$.

On one hand, it is easy to obtain the energy dissipation
\begin{equation}\label{relation1}
\frac{\ud \phi}{\ud t} = \int \pd \phi(h) h_t \ud x = \int - |\pd \phi(h)|^2 \ud x=: -D.
\end{equation}

On the other hand, we use the $\lambda$-convexity of $\phi$ to establish the connection between $D$ and $\phi$. 
First, from the $\lambda$-convexity \eqref{lam-c} with some $\lambda>0$, we know $\phi(h)-\frac{\lambda}{2}\|h\|^2$ is convex. Thus we have
\begin{equation}
\phi(h(t))-\frac{\lambda}{2}h^2(t) - \phi(v) + \frac{\lambda}{2} v^2 \leq \la \pd \phi (h(t))-\lambda h(t), h(t)-v \ra 
\end{equation}
for any $v\in L^2.$
Then we obtain
\begin{equation}\label{relation2}
\phi(h(t))-\phi(v) \leq -\frac{\lambda}{2} \|h(t)-v\|^2 + \la \pd \phi(h(t)), h(t)-v \ra
\end{equation}
for any $v\in L^2.$ Using Young's inequality, we have
\begin{equation}
 \la \pd \phi(h(t)), h(t)-v \ra \leq \frac{\lambda}{2} \|h(t)-v\|^2 + \frac{1}{2\lambda} \|\pd \phi(h(t))\|^2.
\end{equation}
From this, the estimate \eqref{relation2} becomes
\begin{equation}\label{relation3}
\phi(h(t))-\phi(v) \leq -\frac{\lambda}{2} \|h(t)-v\|^2 + \la \pd \phi(h(t)), h(t)-v \ra\leq  \frac{1}{2\lambda} \|\pd \phi(h(t))\|^2.
\end{equation}
Combining \eqref{relation3} with \eqref{relation1}, we obtain
\begin{equation}
\frac{\ud}{\ud t} (\phi(h(t))-\inf \phi) = -\|\pd \phi(h(t))\|^2 \leq -2\lambda(\phi(h(t))-\inf \phi),
\end{equation}
which gives the exponential decay to $\inf \phi = \phi (h^*)$ with $h^*=0$, i.e.
\begin{equation}
\frac{\lambda}{2}\|h(t)-h^*\|^2\leq\phi(h(t))-\phi(h^*) \leq (\phi(h(0))-\phi(h^*))e^{-2\lambda t}.
\end{equation}

\bibliographystyle{alpha}
\bibliography{bibliogr}

\end{document}